\documentclass[12pt]{article}

\usepackage{fullpage}
\usepackage{amsfonts,amsmath,epsf,epsfig,bbm}
\usepackage{amsthm}

\usepackage{tikz}

\newtheorem{theorem}{Theorem}[section]
\newtheorem{definition}[theorem]{Definition}

\newtheorem{corollary}[theorem]{Corollary}
\newtheorem{remark}[theorem]{Remark}

\newtheorem{example}[theorem]{Example}

\newcommand{\VCD}{\mathrm{VCD}}
\newcommand{\RTD}{\mathrm{RTD}}

\newcommand{\sm}{\setminus}
\newcommand{\eset}{\emptyset}
\newcommand{\cM}{{\mathcal M}}
\newcommand{\cC}{{\mathcal C}}
\newcommand{\cS}{{\mathcal S}}

\newcommand{\seq}{\subseteq}
\newcommand{\ra}{\rightarrow}

\newcommand{\dund}{\Leftrightarrow}
\newcommand{\impl}{\Rightarrow}

\let\vec\mathbf

\DeclareMathOperator{\SMN}{SMN}

\DeclareMathOperator{\GMN}{GMN}
\DeclareMathOperator{\dom}{dom}
\DeclareMathOperator{\cost}{cost}
\DeclareMathOperator{\ACN}{AMN}
\DeclareMathOperator{\AN}{AMN}
\DeclareMathOperator{\STD}{STD}

\title{The Hierarchy of Saturating Matching Numbers }

\author{
Hans U. Simon\thanks{Faculty of Mathematics, Ruhr-University Bochum, Germany. E-mail: {\tt hans.simon@rub.de}.}
\and
Jan Arne Telle\thanks{Department of Informatics, University of Bergen, Norway. E-mail: {\tt jan.arne.telle@uib.no}.
}}

\begin{document}

\maketitle

\begin{abstract}
In this paper, we study three matching problems all of which
came up quite recently in the field of machine teaching.
The cost of a matching is defined in such a way that,
for some formal model of teaching, it equals (or bounds)
the number of labeled examples needed to solve a given
teaching task. We show how the cost parameters associated
with these problems depend on each other and how they are
related to other (more or less) well known combinatorial
parameters.
\end{abstract}

\section{Introduction} \label{sec:introduction}

A binary concept class $C$ contains classification rules
(called concepts) which split the instances in its domain
$X = \dom(C)$ into positive examples (labeled $1$) and
negative examples (labeled $0$). Formally, a concept
can be identified with a function $f:c \ra \{0,1\}$.

This paper is inspired by recent work in the field of machine teaching.
In machine teaching, the goal is to find a collection $S$ of helpful
examples, called a \emph{teaching set for the unknown target
concept $c$ in the class $C$}, which enables a learner (or a learning
algorithm) to ``infer'' $c$ from $S$.
The three matching problems that we discuss in this paper are
related to teaching models that came up quite
recently~\cite{FH-OT2022,ST2024,FGHHT2024,MSSZ2022}. We will look at
these problems from a purely combinatorial perspective but,
as a service for the interested reader, we will indicate
the connection to machine teaching by giving a reference
to the relevant literature at the appropriate time.

\section{Preliminaries} \label{sec:preliminaries}

We first recall a few definitions from learning theory.
Let $X$ be a finite set and let $2^X$ denote the set of all
functions $c:X \ra \{0,1\}$. A family $C \seq 2^X$ is called
a \emph{concept class over the domain $X = \dom(C)$}.
A set $S = \{(x_1,b_1),\ldots, (x_m,b_m)\}$ with $m$ distinct
elements $(x_i,b_i) \in X\times\{0,1\}$ is called a
\emph{labeled sample of size $|S| = m$}. A concept $c \in C$
is said to be \emph{consistent with $S$} if $c(x_i) = b_i$
for $i = 1,\ldots,m$. A labeled sample $S$ is said to be
\emph{realizable by $C$} if $C$ contains a concept that 
is consistent with $S$. A \emph{$C$-saturating matching} is
a mapping $M$ that assigns to each $c \in C$ a labeled sample $M(c)$
such that the following holds: first $c$ is consistent with $M(c)$;
second $c \neq c' \in C$ implies that $M(c) \neq M(c')$.
The \emph{cost of $M$} is defined as the size of the largest labeled
sample that is assigned by $M$ to one of the concepts in $C$,
i.e., $\cost(M) = \max_{c \in C}|M(c)|$. 
A $C$-saturating matching $M'$ is called a
\emph{direct improvement} of another $C$-saturating matching $M$
if $M'$ differs from $M$ only on a single concept $c$ 
and $|M'(c)| < |M(c)|$.
A $C$-saturating matching $M$ is called \emph{greedy} if it
does not admit for direct improvements. 

Suppose that $P$ is a procedure that, given a concept class $C$, 
returns a $C$-saturating matching $M$. $P$ is said to be \emph{greedy}
if it inspects the concepts one by one and, for each fixed 
concept $c \in C$, makes an assignment $M(c) = S$ where $S$ is a 
labeled sample with the following properties:
\begin{enumerate}
\item
The concept $c$ is consistent with $S$.
\item
The labeled sample $S$ is still available, i.e., it has not already
been assigned to a different concept.
\item
Among the samples which meet the preceding two criteria, $S$ is
of lowest cost.
\end{enumerate}
A greedy procedure has some degrees of freedom. First, we have
left open in which order the concepts are inspected. Second,
for each concept, $P$ may have the choice between different
labeled samples of the same size. But we could make the output
of a greedy procedure unique by committing to a linear order
among the concepts, and by committing to a linear extension
of the partial size-ordering among the labeled samples.
Without such commitments, $P$ may have several possible outputs.
Greedy $C$-saturating matchings and greedy procedures for their
computation are related as follows:
\begin{remark}
A $C$-saturating matching is greedy iff it is a possible output
of a greedy procedure.
\end{remark}

In the sequel, the symbol $X$ will always denote the domain
of a concept class $C$. The latter will be clear from context
We are now prepared to define the parameters we are interested in.
\begin{description}
\item[SMN:] 
The \emph{saturating matching number of $C$}, denoted by $\SMN(C)$, 
is the cost of a cheapest $C$-saturating matching. This parameter 
was introduced first in \cite{ST2024}. It characterizes the
sample complexity of MAP- resp.~MLE-based teaching. See~\cite{ST2024}
for details.
\item[SMN$'$:]
Here comes a relative of the saturating matching number. 
We define $\SMN'(C)$ as the smallest number $d$ such that the number
of $C$-realizable samples of size at most $d$ is greater than or equal
to $|C|$.
\item[AMN:] 
The \emph{antichain matching number of $C$}, denoted by $\ACN(C)$, is the
cost of a cheapest $C$-saturating matching that has the antichain 
property. The antichain matching number was introduced first 
in~\cite{MSSZ2022}. As explained in~\cite{MSSZ2022}, it lower-bounds
the sample complexity in all models where the teaching map has the
antichain property (or can be modified so as to have this property 
without increasing its cost).
\item[AMN$'$:]
We are also interested in the following relative of the antichain
matching number.
We define $\ACN'(C)$ as the smallest number $d$ such that
the $C$-realizable samples of size at most $d$ contain an antichain 
of size $|C|$.
\item[GMN:]
The parameter $\GMN(C)$, introduced first in~\cite{FGHHT2024}
in connection with learning representations of concepts, is defined 
as the largest possible cost of a greedy $C$-saturating matching. 
It is easy to see that the smallest possible cost of a greedy
$C$-saturating matching equals $\SMN(C)$.
\item[GMN$'$:]
We define $\GMN'(C)$ as the smallest number $d$ such  
that $\sum_{i=0}^{d}\binom{|X|}{i} \ge |C|$.
\item[VCD:]
Points $x_1\ldots,x_r \in X$ are said to be \emph{shattered by $C$}
if, for every choice of $b_1,\ldots,b_r \in \{0,1\}$, the
labeled sample $\{(x_1,b_1),\ldots,(x_r,b_r)\}$ is $C$-realizable.
The size of the largest set that is shattered by $C$ is called
the \emph{VC-dimension of $C$} and denoted by $\VCD(C)$.
The VC-dimension was first introduced in~\cite{VC1971}.
\item[RTD:]
By $\RTD(C)$, we denote the so-called \emph{recursive teaching
dimension of $C$}. We will not require a formal definition of the
RTD in this paper, but the interested reader may 
consult~\cite{ZLHZ2011,DFSZ2014} for a definition.
\item[STD$_{min}$:]
A sequence $\cM = (M_k)_{k\ge0}$ of $C$-saturating matchings 
is called a \emph{subset teaching sequence for $C$} if the following hold:
\begin{enumerate}
\item
$M_0(c) = \{(x,c(x)): x \in X\}$ for all $c \in C$.
\item
$M_{k+1}(c) \seq M_k(c)$ for all $k\ge0$ and all $c \in C$.
\item
$M_{k+1}(c) \not\seq M_k(c')$ for all $k\ge0$ and all $c' \neq c \in C$.
\end{enumerate}
Let 
\[ 
k^*(\cM) = \min\{k| \forall k' \ge k, c \in C: M_{k'}(c) = M_k(c)\}
\enspace .
\] 
We define $\cost(\cM$) as the cost of the matching $M_{k^*(\cM)}$. 
Finally, we set
\[ 
\STD_{min}(C) = \min\{\cost(\cM): \mbox{$\cM$ is a subset teaching
sequence for $C$}\} \enspace .
\]
This parameter was introduced in~\cite{MSSZ2022}. It is a variant 
of Balbach's~\cite{B2008} subset-teaching dimension.
\end{description}

The paper is organized as follows.
In Section~\ref{sec:first-look} and~\ref{sec:kn-family}, we verify 
several $\le$-relations between the combinatorial parameters that 
we are interested in. The resulting parameter hierarchy,
called SMN-GMN-AMN hierarchy in what follows, is  visualized 
in Fig.~\ref{fig:hierarchy} below. 
Class- and domain-monotonicity are desirable properties for parameters
related to teaching. In Section~\ref{sec:monotonicity}, we determine 
which of our combinatorial parameters enjoy these properties.
An evaluation of the parameters on the powerset takes place
in Section~\ref{sec:powerset}. In Section~\ref{sec:additivity},
we determine which of our combinatorial parameters are additive
(resp.~sub- or super-additive) on the free combination of two
concept classes. We show finally in Section~\ref{sec:final-look} 
that all inequalities of the SMN-GMN-AMN hierarchy are occasionally 
strict so that this hierarchy is proper.

\begin{figure}[hbt] 
        \begin{center}
\begin{tikzpicture}


        \tikzset{every node/.style={draw,rectangle}}

        \def\x{0} \def\y{0}

        \node (a) at (\x,\y) {$\min\{\VCD,\RTD\}$};
        \node (b) at (\x-1.3,\y-1) {$\GMN'$};
        \node (c) at (\x-1.3,\y-2) {$\GMN$};
        \node (d) at (\x-1.3,\y-3) {$\SMN$};
        \node (e) at (\x,\y-4) {$\SMN'$};

        \node (f) at (\x+1.3,\y-2) {$\AN$};
        \node (g) at (\x+1.3,\y-3) {$\AN'$};
        \node (h) at (\x+1.3,\y-1) {$\STD_{min}'$};

        \draw[<-] (a) to (h);
        \draw[<-] (a) to (b);
        \draw[<-] (b) to (c);
        \draw[<-] (c) to (d);
        \draw[<-] (d) to (e);
        \draw[<-] (f) to (g);
        \draw[<-] (g) to (e);
        \draw[<-] (b) to (g);
        \draw[<-] (f) to (d);
        \draw[<-] (h) to (f);

\end{tikzpicture}
        \end{center}
        \caption{The parameter hierarchy: an arc $Y \ra Z$ represents
         a relation $Y(C) \le Z(C)$ that is valid for each concept
         class $C$ and that is occasionally strict (for special choices 
         of $C$).  \label{fig:hierarchy} }
\end{figure}
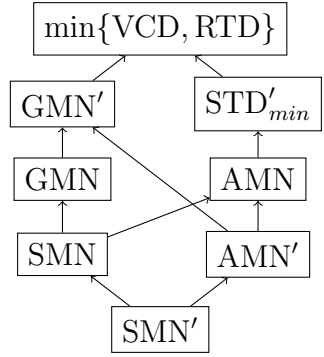

\section{A First Look at the SMN-GMN-AMN Hierarchy} \label{sec:first-look}

Here is a list of the (more or less) obvious relations
among the parameters of the SMN-GMN-AMN hierarchy:
\begin{enumerate}
\item $\SMN(C) \le \GMN(C)$ and $\SMN(C) \le \ACN(C)$. \\
{\bf Reason:} No matching can be cheaper than a cheapest one.
\item $\SMN'(C) \le \SMN(C)$. \\
{\bf Reason:}
Let $M$ be a $C$-saturating matching of cost $\SMN(C)$.
Then the samples $M(c)$ with $c \in C$ are $C$-realizable. 
It follows that the number of $C$-realizable samples of size 
at most $\SMN(C)$ is greater than or equal to $|C|$. The 
definition of $\SMN'(C)$ now implies that $\SMN'(C) \le \SMN'(C)$.
\item $\ACN'(C) \le \ACN(C)$. \\
{\bf Reason:} 
Let $M$ be a $C$-saturating matching of cost $\ACN(C)$ that has the 
antichain property. Then the samples $M(c)$ with $c \in C$ are 
$C$-realizable and form an antichain of size $|C|$. 
The definition of $\ACN'(C)$ now implies that $\ACN'(C) \le \ACN(C)$.
\item $\GMN(C) \le \GMN'(C)$. \\
{\bf Reason:}
Set $d := \GMN'(C)$ and $X := \dom(C)$. It follows 
that $\sum_{i=0}^{d}\binom{|X|}{i} \ge |C|$.
Let $M$ be a $C$-saturating matching of cost $d' > d$. Pick a concept $c_0 \in C$
such that $|M(c_0)| = d'$. For cardinality reasons, there must exist an
unlabeled sample $\{x_1,\ldots,x_r\}$ of size $r \le d$ such that no $c \in C$
satisfies $M(c) = \{(x_1,c(x_1)),\ldots,(x_r,c(x_r))\}$. Let $M'$ be
the matching that equals $M$ except 
for $M'(c_0) =   \{(x_1,c_0(x_1)),\ldots,(x_r,c_0(x_r))\}$.
Then $M'$ is a $C$-saturating matching that directly improves on $M$.
We have therefore shown that no $C$-saturating matching of cost exceeding $d$
is greedy. Hence $\GMN(C) \le \GMN'(C)$.
\item $\SMN'(C) \le \ACN'(C)$ \\
{\bf Reason:}
If the $C$-realizable samples of size at most $d$ contain an antichain
of size $|C|$ then, obviously, the number of $C$-realizable samples of size at 
most $d$ cannot be smaller than $|C|$. The inequality $\SMN'(C) \le \AN'(C)$ 
is therefore evident from the definitions of $\SMN'(C)$ and of $\AN'(C)$.
\item 
$\AN(C) \le \STD_{min}(C)$ and $\STD_{min}(C) \le \min\{\VCD(C) , \RTD(C)\}$. \\
{\bf Reason:}
It is obvious from the definition of a subset teaching 
sequence $\cM = (M_k)_{k\ge0}$ (and it was noted in~\cite{MSSZ2022}
already) that each matching $M_k$ is $C$-saturating and has the 
antichain property. Hence $\AN(C) \le \STD_{min}(C)$. It was shown 
furthermore in~\cite{MSSZ2022} that $\STD_{min}(C)$ cannot exceed $\VCD(C)$ 
or $\RTD(C)$. 
\item $\GMN'(C) \le \min\{\VCD(C),\RTD(C)\}$. \\
{\bf Reason:}
According to Sauer's Lemma~\cite{S1972,Sh1972}, we have 
that $\sum_{i=0}^{\VCD(C)}\binom{|X|}{i} \ge |C|$.
This inequality holds as well with $\RTD(C)$ in place 
of $\VCD(C)$~\cite{SSYZ2014}. The inequalities $\GMN'(C) \le \VCD(C)$ 
and $\GMN'(C) \le \RTD(C)$ are now immediate from the definition
of $\GMN'(C)$.
\item
If $\SMN'(C) \le |X|/5$, then $\GMN'(C) \le 2\cdot\SMN'(C)$. \\
{\bf Reason:}
For sake of brevity, set $n = |X|$. Let $d$ be the smallest number 
subject to $2^d\cdot\sum_{i=0}^{d}\binom{n}{i} \ge |C|$. Since 
the number of $C$-realizable samples of size at most $d$ is bounded from above 
by $\sum_{i=0}^{d}2^i\cdot\binom{n}{i} \le 2^d\cdot\sum_{i=0}^{d}\binom{n}{i}$,
it follows from the definition of $\SMN'(C)$ that $\SMN'(C) \ge d$.
It suffices therefore to show $\GMN'(C) \le 2d$. We will show that 
\begin{equation} \label{eq1:binomial-coefficients}
\sum_{i=0}^{2d}\binom{n}{i} \ge 2^d\cdot\sum_{i=0}^{d}\binom{n}{i} 
\enspace . 
\end{equation}
Thanks to $2^d\cdot\sum_{i=0}^{d}\binom{n}{i} \ge |C|$, this would imply 
that $\GMN'(C) \le 2d$. We verify~(\ref{eq1:binomial-coefficients})
by proving that
\[ 
\forall i=0,\ldots,d: \frac{\binom{n}{d+i}}{\binom{n}{i}} \ge 2^d \enspace .
\]
Expanding the binomial coefficients, we obtain (after some cancellation)
\[
\frac{\binom{n}{d+i}}{\binom{n}{i}} = 
\frac{(n-i)(n-i-1)\ldots(n-i-d+1)}{(d+i)(d+i-1)\ldots(i+1)}
= \frac{n-i}{d+i} \cdot \frac{n-i-1}{d+i-1} \cdot\ldots\cdot \frac{n-i-d+1}{i+1}
\enspace .
\]
Given our assumption that $d \le n/5$, each of the $d$ factors in the latter 
product has a value of at least $2$, which accomplishes the proof.
\item If $\SMN'(C) \le |X|/5$, then $\GMN(C) \le 2\cdot\SMN(C)$. \\
{\bf Reason:}
This follows from inequalities that have been verified already:
$\GMN(C) \le \GMN'(C) \le 2\cdot\SMN'(C) \le 2\cdot\SMN(C)$. 
\end{enumerate}

\section{Bounding AMN$'$ by GMN$'$} \label{sec:kn-family}

The verification of the inquality $\AN'(C) \le \GMN'(C)$ is 
still missing and will be given in the course of this section.

The \emph{$(k,n)$-family} is defined as the family of concept 
classes $C$ with $|C| = k$ and $|\dom(C)| = n$. We denote this family
by $\cC_{k,n}$. 

\begin{example}
For $k \ge 2$ and $n \ge \log k$, we denote by $C_{k,n}$ the 
concept class in $\cC_{k,n}$ given by
\[ 
\forall i = 0,1,\ldots,k-1 , j = 0,1,\ldots,n-1:
i = \sum_{j=0}^{n-1}c_i(x_j)2^j \enspace .
\]
In other words, $C_{k,n} = \{c_0,c_1,\ldots,c_{k-1}\}$,
$\dom(C_{k,n}) = \{x_0,x_1,\ldots,x_{n-1}\}$
and $c_i(x_j)$ is the $j$-th bit in the binary representation of $i$.
\end{example}

\noindent
It is easy to characterize the $C_{k,n}$-realizable samples:
\begin{remark} \label{rem:minimum-realizability}
A labeled sample $S$ over $\dom(C_{k,n})$ is $C_{k,n}$-realizable
iff $\sum_{j:(x_j,1) \in S}2^j \le k-1$.
\end{remark}

According to the following quite recent result, 
the class $C_{k,n} \in \cC_{k,n}$ has the smallest number
of realizable samples of a given size: 

\begin{theorem}[Main Theorem in~\cite{SHKT2024}] \label{th:hart-general}
For each $1 \le d \le n$, the number of $C$-realizable samples 
of size $d$, with $C$ ranging over all concept classes in $\cC_{k,n}$,
is minimized by setting $C = C_{k,n}$.\footnote{The main theorem 
in~\cite{SHKT2024} is not stated in this form, but it is equivalent 
to what is claimed in Theorem~\ref{th:hart-general}.}
\end{theorem}

Since the minimizer $C_{k,n}$ of the number of realizable samples
of size $d$ is the same for all possible choices of $d$, we immediately
obtain the following result:

\begin{corollary} \label{cor:smn-maximizer}
$\SMN'(C)$ with $C$ ranging over all concept classes in $\cC_{k,n}$,
is maximized by setting $C = C_{k,n}$.
\end{corollary}

\noindent
Here is another (less immediate than Corollary~\ref{cor:smn-maximizer})
application of Theorem~\ref{th:hart-general}:

\begin{theorem} \label{th:challenging-relation}
For each concept class $C$, we have that $\AN'(C) \le \GMN'(C)$.
\end{theorem}

\begin{proof}
For each concept class $C$, let $\AN''(C)$ be the smallest number $d$ 
such that there exist $|C|$ many $C$-realizable labeled samples of size $d$. 
Since different sets of the same size always form an antichain,
it follows that  $\AN'(C) \le \AN''(C)$. It suffices therefore to show that,
for each concept class $C$, we have that $\AN''(C) \le \GMN'(C)$. 
Let $C_0$ be an arbitrary but fixed concept class. We set 
\[
k := |C_0|\ ,\ \ell := \lfloor \log k \rfloor\ \mbox{ and }\  n := |\dom(C_0)|
\enspace .
\] 
Note that $k \le 2^n$ because there can be at most $2^n$
distinct concepts over a domain of size $n$. By definition, $\GMN'(C_0)$ 
is the smallest number $d$ such that $\sum_{i=0}^{d}\binom{n}{i} \ge k$. 
Note that $\GMN'(C)$ is the same number for all concept classes from the 
family $C_{k,n}$.
Specifically $\GMN'(C_0) = \GMN'(C_{k,n})$. On the other hand,
it is immediate from Theorem~\ref{th:hart-general} (and again the fact 
that different sets of the same size always form an antichain) that
\[ 
\AN''(C_{k,n}) = \max\{\AN''(C): C \in \cC_{k,n}\} \enspace .
\]
Thus, if we knew that $\AN''(C_{k,n}) \le \GMN'(C_{k,n})$, 
we could conclude that
\[ 
\AN''(C_0) \le \AN''(C_{k,n}) \le \GMN'(C_{k,n}) = \GMN'(C_0)
\enspace .
\]
It suffices therefore to verify the 
inequality $\AN''(C_{k,n}) \le \GMN'(C_{k,n})$.
We will make use of the following auxiliary result.
\begin{description}
\item[Claim 1:] 
$\GMN'(C_{k,n}) \le \ell$.
\item[Proof of Claim 1:]
Since $k \le 2^n$, it follows that $n \ge \lceil \log k \rceil$. 
The function $n \mapsto \sum_{i=0}^{\ell}\binom{n}{i}$ is monotonically
increasing with $n$. But even for $n = \lceil \log k \rceil$,
we have that 
\[
S := \sum_{i=0}^{\ell}\binom{\lceil \log k \rceil}{i} \ge k
\]
for the following reason:
\begin{itemize}
\item 
If $k$ is a power of $2$, then $\ell = \log k = \lceil \log k \rceil$
and, therefore, $S = 2^\ell = k$.
\item
If $k$ is not a power of $2$, then $\ell = \lfloor \log k \rfloor$
and $\lceil \log k \rceil = \ell+1$. This implies 
that $S = 2^{\ell+1} - 1 > 2^{\log k}-1 = k-1$ and, because $S$ is 
an integer, it implies that $S \ge k$.  
\end{itemize}
It follows from this discussion that $\GMN'(C_{k,n}) \le \ell$.
\item[Claim 2:]
For each $d \in [\ell]$, there exists an injective mapping $f$ 
which transforms an unlabeled sample of size at most $d$ 
over domain\footnote{We talk here about the domain of $C_{k,n}$.} 
$X = \{x_0,x_1,\ldots,x_{n-1}\}$ into a (labeled) $C_{k,n}$-realizable 
sample of size exactly $d$.
\item[Proof of Claim 2:]
Let $U \seq X$ be an unlabeled sample of size at most $d$.
We define $f(U)$ as the (initially empty) labeled sample that 
is obtained from $U$ as follows:
\begin{enumerate}
\item 
For each $x_j \in U$ with $\ell \le j \le n-1$, insert $(x_j,0)$ 
into $f(U)$.
\item
For each $x_j \in U$ with $0 \le j \le \ell-1$, insert $(x_j,1)$ 
into $f(U)$.
\item
While $|f(U)| < d$, pick some instance $x_j$ 
from $\{x_0,\ldots,x_{\ell-1}\} \sm f(U)$ and insert $(x_j,0)$
into $f(U)$. 
\end{enumerate}
It is obvious that, after Step 3, the set $f(U)$ is of size $d$.
Moreover 
\[
\sum_{j:(x_j,1) \in U}2^j \le \sum_{j=0}^{\ell-1}2^j = 2^\ell-1 \le k-1
\enspace ,
\]
which, according to Remark~\ref{rem:minimum-realizability},
implies that $f(U)$ is $C_{k,n}$-realizable. Finally
observe that $U$ can be reconstructed from $f(U)$:
\[
U = \{x_j: (\ell \le j \le n-1 \wedge (x_j,0) \in f(U)) \vee
(0 \le j \le \ell-1 \wedge (x_j,1) \in f(U)\} \enspace .
\]
Hence $f$ is injective, which completes the proof of Claim 2.
\end{description}
The proof of Theorem~\ref{th:challenging-relation} can now be 
accomplished as follows. Set $d := \GMN'(C_{k,n})$. Then there 
exist at least $k$ distinct unlabeled samples of size at most $d$ 
over domain $X$.  Thanks to Claim 1, we know that $d \le \ell$. 
Thanks to Claim 2, we may now conclude that there exist at least $k$ 
distinct labeled $C_{k,n}$-realizable samples over domain $X$,
each of size exactly $d$. This shows that $\AN''(C_{k,n}) \le d$.
\end{proof}

\section{Monotonicity Considerations} \label{sec:monotonicity}

We say that $(C',X')$ is an \emph{extension of $(C,X)$} if $C \seq C'$ 
and $X \seq X'$. If $C \subset C'$ and $X=X'$, it is called a
\emph{class extension}. If $X \subset X'$, $C = \{c_{|X}: c \in C'\}$ 
and $|C| = |C'|$, it is called a \emph{domain extension}. Note that
the condition $|C| = |C'|$ ensures that $X$ \emph{distinguishes between 
the concepts in $C'$}, i.e., 
\begin{equation} \label{eq:domain-extension}
\forall c^1,c^2 \in C': c^1 \neq c^2 \impl c^1_{|X} \neq c^2_{|X}
\enspace .
\end{equation} 
A combinatorial parameter $Y(C)$, associated with a concept class $C$, 
is said to be \emph{class monotonic} if, for each class extension $(C',X)$ 
of $(C,X)$, we have that $Y(C) \le Y(C')$. It is said to be \emph{domain 
monotonic} if, for each domain extension $(C',X')$ of $(C,X)$, we have 
that $Y(C') \le Y(C)$. 

\begin{theorem} \label{th:monotonicity}
The monotonicity properties of the parameters in the SMN-GMN-AMN hierarchy
are as it is shown in Table~\ref{table:monotonicity} below.\footnote{
$\VCD$ is domain monotonic in the other direction: for each domain
extension $(C',X')$ of $(C,X)$, we have that $\VCD(C) \le \VCD(C')$.}
\end{theorem}

\begin{table} 
\begin{center}
\begin{tabular}{|c||c|c|c|} 
\hline
& class monotonic & domain monotonic & Reference \\
\hline
\hline
$\VCD$ & yes & no & common knowledge \\
\hline
$\RTD$ & yes & yes & \cite{DFSZ2014} \\
\hline
$\min\{\VCD,\RTD\}$ & yes & no & this paper \\
\hline
$\STD_{min}$ & yes & yes & \cite{MSSZ2022} \\
\hline
$\AN$ & yes & yes & this paper \\
\hline
$\AN'$ & no & yes & this paper \\ 
\hline
$\GMN'$ & yes & yes & this paper \\
\hline
$\GMN$ & yes & yes & this paper \\
\hline
$\SMN$ & yes & yes & this paper \\
\hline
$\SMN'$ & no & yes & this paper \\
\hline
\end{tabular}
\caption{Monoconicity properties of the parameters of the SMN-GMN-AMN 
hierarchy.}
\label{table:monotonicity}
\end{center}
\end{table}

\begin{proof}
The monotonicity properties of $\VCD$, $\RTD$ and $\STD_{min}$
are well known. We will now verify the remaining entries in 
Table~\ref{table:monotonicity} row by row.
\begin{description}
\item[min\{VCD,RTD\}:]
Set $Y = \min\{\VCD,\RTD\}$. For trivial reasons, the following implication 
is valid: if $Z_1$ and $Z_2$ are class (resp.~domain) monotonic, 
then $\min\{Z_1,Z_2\}$ is class (resp.~domain) monotonic. 
Setting $Z_1 = \VCD$ and $Z_2 = \RTD$, it follows that $Y$ is class 
monotonic. \\
In order to show that $Y$ is not domain monotonic, we will make use
of the class $C'$ over domain $X' = \{x_1,\ldots,x_5,x_6\}$ that
is given by the following table:
\[
\begin{array}{l|cccccc}
C'& x_1 & x_2 & x_3 & x_4 & x_5 & x_6 \\
\hline
c_1 & {\mathbf 1} & \mathbf{1} & 0 & 0 & 0 & {\mathbf 0} \\
c_2 & {\mathbf 0} & {\mathbf 1} & 1 & 0 & 0 & {\mathbf 0} \\
c_3 & {\mathbf 0} & {\mathbf 0} & 1 & 1 & 0 & {\mathbf 1} \\
c_4 & {\mathbf 0} & {\mathbf 0} & 0 & 1 & 1 & {\mathbf 0} \\
c_5 & {\mathbf 1} & {\mathbf 0} & 0 & 0 & 1 & {\mathbf 1} \\
c_6 & 0 & 1 & 0 & 1 & 1 & 0 \\
c_7 & {\mathbf 0} & {\mathbf 1} & 1 & 0 & 1 & {\mathbf 1} \\
c_8 & 1 & 0 & 1 & 0 & 1 & 1 \\
c_9 & {\mathbf 1} & {\mathbf 0} & 1 & 1 & 0 & {\mathbf 0} \\
c_{10} & {\mathbf 1} & {\mathbf 1} & 0 & 1 & 0 & {\mathbf 1} 
\end{array}
\]
The class over subdomain $X = \{x_1,\ldots,x_5\}$ which is given
by the first 5 columns in this table is known under the name
Warmuth's class\footnote{This class was designed by Manfred Warmuth
so as to provide a simple class for which the $\RTD$ exceeds the $\VCD$.}
and denoted by $C_{MW}$. Clearly $(C',X')$ is a domain extension 
of $(C_{MW},X)$. The class $C_{MW}$ is known 
to satisfy $\VCD(C_{MW}) = 2$ and $\RTD(C_{MW}) = 3$.
\begin{description} 
\item[Claim:] $\VCD(C') = \RTD(C') = 3$.
\item[Proof:]
Clearly $\VCD(C') \le 3$ because $\VCD(C) = 2$ and extending the domain
by one instance (here: $x_6$) can increase the VC-dimension at most by $1$. 
An inspection of the table-entries highlighted in bold reveals
that $\VCD(C') \ge 3$. Hence \hbox{$\VCD(C')=3$}. \\
Since $\RTD$ is domain monotonic, we have that $\RTD(C') \le \RTD(C) = 3$.
Assume for contradiction that $\RTD(C') \le 2$. Then one of the subclasses 
\begin{eqnarray*}
C'_0 & = & \{c \in C': c(x_6) = 0\} = \{c_1,c_2,c_4,c_6,c_9\} \\
C'_1 & = & \{c \in C': c(x_6) = 1\} = \{c_3,c_5,c_7,c_8,c_{10}\}
\end{eqnarray*}  
must contain a concept which can be uniquely specified by one more
labeled example. An inspection of the following tables for $C'_0$
and $C'_1$ reveals that this is impossible so that we arrived at
a contradiction:
\[
\begin{array}{l|cccccc}
C'_0 & x_1 & x_2 & x_3 & x_4 & x_5 & x_6 \\
\hline
c_1 & 1 & 1 & 0 & 0 & 0 & 0 \\
c_2 & 0 & 1 & 1 & 0 & 0 & 0 \\
c_4 & 0 & 0 & 0 & 1 & 1 & 0 \\
c_6 & 0 & 1 & 0 & 1 & 1 & 0 \\
c_9 & 1 & 0 & 1 & 1 & 0 & 0 \\
\end{array}
\hspace{1cm}
\begin{array}{l|cccccc}
C'_1& x_1 & x_2 & x_3 & x_4 & x_5 & x_6 \\
\hline
c_3 & 0 & 0 & 1 & 1 & 0 & 1 \\
c_5 & 1 & 0 & 0 & 0 & 1 & 1 \\
c_7 & 0 & 1 & 1 & 0 & 1 & 1 \\
c_8 & 1 & 0 & 1 & 0 & 1 & 1 \\
c_{10} & 1 & 1 & 0 & 1 & 0 & 1 
\end{array}
\]
\end{description}
We conclude from this discussion that $Y(C) = 2 < 3 = Y(C')$.
Thus $Y$ is not domain monotonic.\footnote{It is easy to find 
pairs $(C',X')$ and $(C,X)$ such that $(C',X')$ is a domain
extension of $(C,X)$ and $Y(C) > Y(C')$. Thus extending the domain
can have effects in both directions: it can lead to a greater value
of $Y = \min\{\VCD,\RTD\}$, but it can also lead to a smaller value.}
\item[AMN:] 
Let $(C',X)$ be a class extension of $(C,X)$.
Let $M$ be a $C'$-saturating matching of cost $d$ that has the antichain
property. Then $M_{|C}$ is a $C$-saturating matching of cost at most $d$ 
that has the antichain property. This shows that $\AN$ is class monotonic. \\
Let $(C',X')$ be a domain extension of $(C,X)$.
Let $M$ be a $C$-saturating matching of cost $d$ that has the antichain
property. Let $M'$ be the mapping given by $M'(c') = M(c'_{|X})$. 
Because of~(\ref{eq:domain-extension}) $M'$ is a $C'$-saturating matching.
Moreover $M'$ is of the same cost as $M$ and inherits the antichain property
from $M$. This shows that $\AN$ is domain monotonic.
\item[AMN$'$:]
The following example shows that $\AN'$ is not class monotonic.
Consider the concept class $C = \vec{C}_{16,9}$. It has $18$ samples 
of size $1$ but only $13$ of them are realizable by~$C$. Since $13 < 16 = |C|$, 
we get $\AN'(C) \ge 2$. Let $C'$ be the class obtained from $C$ by adding
the all-ones concept. $C'$ is of size $17$ but now all $18$ samples of
size $1$ are realizable (and, of course, they form an antichain).
Hence $\AN'(C') = 1$, which is in contradiction with class monotonicity. \\
Let $(C',X')$ be a domain extension of $(C,X)$.
Remember that this implies that $|C| = |C'|$.
Let $A$ be an antichain of size $|C|$ consisting of $C$-realizable 
samples of size at most $d$ over domain $X$. Then $A$ is also an antichain 
of size $|C'|$ consisting of $C'$-realizable samples of size at most $d$
over domain $X'$ (albeit none of these samples contains a point
from $X' \sm X$). This shows that $\AN'$ is domain monotonic.
\item[GMN$'$:]
Let $(C',X)$ be a class extension of $(C,X)$. Then $|C'| > |C|$.
Clearly $\sum_{i=0}^{d}\binom{|X|}{i} \ge |C|$ 
if $\sum_{i=0}^{d}\binom{|X|}{i} \ge |C'|$. Thus $\GMN'$ is class 
monotonic. \\
Let $(C',X')$ be a domain extension of $(C,X)$, which implies 
that $|C'| = |C|$ and $|X'| > |X|$. 
Clearly $\sum_{i=0}^{d}\binom{|X'|}{i} \ge |C'|$
if $\sum_{i=0}^{d}\binom{|X|}{i} \ge |C|$. 
This shows that $\GMN'$ is domain monotonic.
\item[GMN:]
Let $(C',X)$ be a class extension of $(C,X)$. 
Let $M$ be a greedy $C$-saturating matching of cost $d$.
In order to prove class monotonicity, it suffices to specify a greedy
$C'$-saturating matching $M'$ of cost at least $d$.
To this end, we consider the following $C'$-saturating matching $M'$:
\begin{enumerate}
\item 
For each $c \in C$, set $M'(c) = M(c)$.
\item
Inspect the concepts $c' \in C' \sm C$ one by one. 
For each fixed $c' \in C' \sm C$, choose a shortest
labeled sample $S'$ among the ones which are consistent with $c'$ 
and have not already been used for another concept before.
Set $M(c') = S'$.
\end{enumerate}
Clearly the cost of $M'$ is lower-bounded by the cost of $M$.
By construction, $M'$ cannot be directly improved on a 
concept $c' \in C' \sm C$. The same remark applies to concepts $c \in C$,
because any direct improvement of $M'$ on $c$ would be a direct
improvement of $M$ of $C$, which is impossible because $M$ is greedy.
It follows from this discussion that $M'$ is a greedy $C'$-saturating
matching of cost at least $d$. This shows that $\GMN$ is class monotonic. \\
Let $(C',X')$ be a domain extension of $(C,X)$, which implies 
that $X \subset X'$. For each $c' \in C'$, let $c = c'_{|X}$ denote 
the corresponding concept in $C$. Arbitrarily fix 
a linear ordering $c'_1,c'_2,c'_3 \ldots$ of the concepts in $C'$ as well as
a linear extension $S_1,S_2,S_3,\ldots$ of the partial size-ordering 
of all $C'$-realizable samples over $X'$. Let $M'$ be the resulting 
greedy $C'$-saturating matching and let $d'$ denote its cost. In order 
to prove domain monotonicity, it suffices show that the greedy procedure 
applied to $C$ has a possible output of cost at least $d'$. 
To this end, we choose the ordering $c_1,c_2,c_3 \ldots$ for the concepts 
in $C$. Furthermore, our ordering of the $C$-realizable samples over $X$ (which 
are also $C'$-realizable!) is obtained from the ordering $S_1,S_2,S_3,\ldots$
by the removal of all samples containing one or more instances from $X' \sm X$.
Let $M$ be the resulting greedy $C$-saturating matching.
It is easy to see that the following holds for each $i \in [m]$:
if $M(c_i) = S_j$ and $M'(c'_i) = S_{j'}$, then $j' \le j$.
This clearly implies that the cost of $M$ ís not less than the cost $d'$
of $M'$. Hence $\GMN$ is domain monotonic.
\item[SMN:]
Let $(C',X)$ be a class extension of $(C,X)$. If $M'$ is a $C'$-saturating
matching of cost $d'$, then $M = M_{|C}$ is a $C$-saturating matching of 
cost $d \le d'$. This implies that $\SMN$ is class monotonic. \\
Let $(C',X')$ be a domain extension of $(C,X)$. If $M$ is a $C$-saturating
matching, then $M'$ with $M'(c') = M(c)$ for $c = c'_{|X}$
is a $C'$-saturating matching of the same cost.
This implies that $\SMN$ is domain monotonic.
\item[SMN$'$:]
The same example, $C = \vec{C}_{16,9}$, that we used above for showing 
that $\AN'$ is not class monotonic can also be used for showing that $\SMN'$ 
is not class monotonic. There are $19$ samples of size at most $1$ (including
the empty sample) but only 14 of them are realizable by $C$. 
Since $14 < 16 = |C|$, we have that $\SMN'(C) \ge 2$. Let again $C'$ be the 
class obtained from $C$ by adding the all-ones concept. $C'$ is of size $17$ 
but now all $19$ samples of size at most $1$ are realizable.
Hence $\SMN'(C') = 1$, which is in contradiction with class monotonicity. \\
Let $(C',X')$ be a domain extension of $(C,X)$. Each $C$-realizable
sample over domain $X$ is also a (very special) $C'$-realizable sample
over domain $X'$. This implies that $\SMN'(C') \le \SMN(C)$.
Thus $\SMN'$ is domain monotonic.
\end{description}
This completes the proof of Theorem~\ref{th:monotonicity}.
\end{proof}

\section{Evaluation of the Parameters on the Powerset}
\label{sec:powerset}

The combinatorial parameters that we have studied so far are now evaluated
for a special concept class, namely the powerset over the 
domain $X = \{x_1,\ldots,x_n\}$. This class will be denoted by $P_n$
in what follows. Note that $|P_n| = 2^n$, $\VCD(P_n) = n$ and every sample 
can be realized by $P_n$. It is also known that $\RTD(P_n)=n$~\cite{DFSZ2014} 
and that $\STD_{min}(P_n) = n$~\cite{MSSZ2022}. Here are some more observations:
\begin{enumerate} 
\item
$\AN(P_n) = \AN'(P_n) = 
\min\left\{d: 2^d\binom{n}{d} \ge 2^n\right\}$. \\
{\bf Reason:}
As shown in~\cite{MSSZ2022}, the maximum antichain contained in the
set of samples of size at most $d$ is formed by the set of labeled 
samples of size exactly $d$. This antichain has size $2^d\cdot\binom{n}{d}$.
This gives the second equation. It was furthermore shown in~\cite{MSSZ2022}
that there exists a $P_n$-saturating matching which assigns to each concept
a labeled sample of size~$\AN'(C)$. This gives the first equation.
\item
For all but finitely many $n$, we have 
that $0.22 \cdot n < \AN(P_n) < 0.23 \cdot  n$. \\
This was also shown in~\cite{MSSZ2022}.
\item 
$\SMN'(P_n) = \min\left\{d: \sum_{i=0}^{d}2^i\binom{n}{i} \ge 2^n\right\}$. \\
{\bf Reason:}
This is immediate from $n = |X|$, $|P_n| = 2^n$, the fact 
that $\sum_{i=0}^{d}2^i\binom{n}{i}$ is the number of samples of size
at most $d$ (all of which are realizable by $P_n$) and the definition
of the parameter $\SMN'$.
\item
$\GMN'(P_n) = \min\left\{d: \sum_{i=0}^{d}\binom{n}{i} \ge 2^n\right\} = n$. \\
{\bf Reason:}
The first equation is immediate from $n = |X|$, $|P_n| = 2^n$ and 
the definition of the parameter $\GMN'$. The second equation is immediate
from the binomial theorem.
\end{enumerate}

As for the powerset, the antichain matching number and the saturating matching
number are very close relatives:

\begin{theorem} \label{th:sandwich-an-smn}
For all but finitely many $n$, we have  
that $\AN(P_n)-1 \le \SMN'(P_n) \le \SMN(P_n) \le \AN(P_n)$. 
\end{theorem}

\begin{proof}
As shown before, the inequalities $\SMN'(C) \le \SMN(C) \le \AN(C)$ are 
valid even for arbitrary concept classes. It suffices therefore to show 
that $\SMN'(P_n) \ge \AN(P_n)-1$, or equivalently, that $\AN(P_n) \le \SMN'(P_n)+1$. 
Set $D = [n]$. Since 
\begin{equation} \label{eq:min-smn}
\SMN'(P_n) = \min\left\{d \in D: \sum_{i=0}^{d}2^i\binom{n}{i} \ge 2^n\right\}
\end{equation}
and
\begin{equation} \label{eq:min-an}
\AN(P_n) = \AN'(P_n) =  
\min\left\{d \in D: 2^d\binom{n}{d} \ge 2^n\right\} \enspace ,
\end{equation} 
it suffices to show that 
\begin{equation} \label{eq:an-smn}
2^{d+1}\binom{n}{d+1} \ge \sum_{i=0}^{d}2^i\binom{n}{i} \enspace .
\end{equation}
To this end, we define
\[
\Phi_d(n) = \sum_{i=0}^{d}\binom{n}{i}\ \mbox{ and }\ 
t(n) = \left\{ \begin{array}{cc}
           \lceil n/3 \rceil & \mbox{if $n \le 12$} \\
           \lfloor n/3 \rfloor + 1 & \mbox{if $n > 12$}
      \end{array} \right.  \enspace . 
\]
As proven in~\cite{KS2000a}, the term $\binom{n}{d}$ is related 
to $\Phi_{d-1}(n)$ as follows: 
\[
\forall d = 1,\ldots, t(n): \binom{n}{d} \ge \Phi_{d-1}(n)\ 
\mbox{ and }\ \forall d = t(n)+1,\ldots,n: \binom{n}{d} < \Phi_{d-1}(n)
\enspace .
\]
Under the assumption that $d+1 \le t(n)$, we get $\binom{n}{d+1} \ge \Phi_d(n)$
so that
\[
2^{d+1}\binom{n}{d+1} \ge 2^{d+1}\Phi_d(n) = 
2^{d+1} \cdot \sum_{i=0}^{d}\binom{n}{i} > \sum_{i=0}^{d}2^i\binom{n}{i}
\enspace .
\]
The verification of~(\ref{eq:an-smn}) will be complete if we can 
justify the assumption that $d+1 \le t(n)$. This is where asymptotics
comes into play. We know that, for all but finitely many $n$, we have 
that $\SMN'(P_n) \le \AN(P_n) < 0.23 \cdot n < n/3 < t(n)$. 
If follows that, for all but finitely many $n$, the equations~(\ref{eq:min-smn}) 
and~(\ref{eq:min-an}) are still valid when we set $D = [1:t(n)-1]$. Hence 
it is justified to assume that $d+1 \le t(n)$.
\end{proof}

By a straightforward, but tedious, calculation, one can verify that
the inequalities claimed in Theorem~\ref{th:sandwich-an-smn} are valid 
for all $n \ge 15$.

\begin{theorem} \label{th:powerset-gmn}
Let $H(p) = p\cdot\log(p) - (1-p)\cdot\log(1-p)$ be the binary entropy function 
and let $p_*$ be the unique solution to the equation $H(p) = 1-2p$ subject
to $p < 1/2$, and let $0 < p_0 < p_*$. 
Then $\GMN(P_n) \le \lceil (1-p_0) \cdot n \rceil$ 
holds for all but finitely many $n$.\footnote{numerical calculations
reveal that $p_* > 0.17$ so that we can choose $p_0 = 0.17$ and
obtain $\GMN(P_n) \le \lceil 0.83 n \rceil$.} 
\end{theorem}

\begin{proof}
Fix a linear ordering on the concepts in $P_n$ and a linear extension
of the partial size-ordering on the labeled samples over $X$ such that
the resulting $P_n$-saturating greedy matching $M$ has cost $\GMN(P_n)$. 
We can think of $M$ as being built in stages: 
in stage $0 \le j \le \GMN(P_n)$, each sample 
of size $j$ obtains one of the concepts in $C$ as its $M$-partner.
Let $\cS_r$ (resp.~$\cS_{\le r}$) denote the set of samples of size $r$ 
(resp.~of size at most $r$). We define
\[
r_* := \max\left\{r: \sum_{j=0}^{r}2^j\cdot\binom{n}{j}\le  2^{n-r} \right\}
\enspace .
\]
The proof of the theorem will be based on the following two claims.
\begin{description}
\item[Claim 1:]
Every sample in $\cS_{\le r_*}$ gets an $M$-partner.
\item[Proof of Claim 1:]
Let's look at the moment in which an $M$-partner 
for a labeled sample $S \in \cS_{\le r_*}$ has to be found.
The number of concepts already matched at this moment is upper-bounded by 
$|\cS_{\le r_*} \sm \{S\}| \le \sum_{j=0}^{r_*}2^j\cdot\binom{n}{j}-1 \le 2^{n-r_*}-1$.
Hence at least one of the at least $2^{n-r_*}$ many neighbors of $S$
must still be unmatched so that $S$ will certainly obtain an $M$-partner.
Since this reasoning applies to any concept in $\cS_{\le r_*}$,
Claim~1 follows.
\item[Claim 2:]
Suppose that $r$ is chosen such that every sample in $\cS_{\le r}$ 
gets an $M$-partner. Suppose that $q$ is chosen such that 
\begin{equation} \label{eq:q0}
\sum_{i=0}^{q-1}\binom{n}{i} + \sum_{j=0}^{r}\binom{n}{j} 
\le \sum_{j=0}^{r}2^j\cdot\binom{n}{j} \enspace .
\end{equation}
Then $\GMN(P_n) \le n-q$. In particular $\GMN(P_n) \le n-r-1$.
\item[Proof of Claim 2:]
We set $d = n-q$. We have to show that every concept $c \in P_n$ obtains 
an $M$-partner during stages $1,\ldots,d$. For reasons of symmetry, it 
suffices to prove this for the concept $c_0 = X$ (the concept assigning 
label $1$ to every instance in $X$). In the remainder of the proof 
of Claim 2, ``sample'' always means ``sample of size at most~$d$''.
We say a labeled sample is of \emph{type A} if it contains only 
$1$-labeled instances. The remaining samples are said to be of \emph{type B}. 
The samples of type A are precisely the ones which are consistent with $c_0$. 
We have the following situation right after stage~$r$:
\begin{itemize}
\item
The number of concepts with an M-partner of type B equals the number of 
samples of type B within $\cS_{\le r}$, and the latter number equals
$\sum_{j=0}^{r}2^j\binom{n}{j} - \sum_{j=0}^{r}\binom{n}{j}$. 
\item
$E_1 := 2^n - \left(\sum_{j=0}^{r}2^j\binom{n}{j} - \sum_{j=0}^{r}\binom{n}{j}\right)$
is the number of the remaining concepts.
\end{itemize}
The number of samples of type A clearly equals 
\begin{equation} \label{eq:type-A-concepts} 
E_2 := \sum_{j=0}^{d}\binom{n}{j} = 2^n -\sum_{j=d+1}^{n}\binom{n}{j} =
2^n -\sum_{j=0}^{n-d-1}\binom{n}{j} =\ 2^n -\sum_{j=0}^{q-1}\binom{n}{j} 
\enspace .
\end{equation}
Note that $E_1 \le E_2$ is equivalent to~(\ref{eq:q0}). Since $q$, by assumption,
is chosen such that the condition~(\ref{eq:q0}) is satisfied, we may
conclude that $E_1 \le E_2$. What does $E_1 \le E_2$ mean? 
It means that the number of concepts in $P_n \sm \{c_0\}$ not being matched 
with a labeled sample of type B is smaller than the number of samples of type A.
Since $c_0$ is adjacent to all samples of type A, this implies that $c_0$ 
will obtain an $M$-partner in stage $d$ (or even in an earlier stage). 
Hence $\GMN(P_n) \le d = n-q$, Finally note that~(\ref{eq:q0}) can be satisfied
by setting $q = r+1$. Hence $\GMN(P_n) \le n-q-1$, which completes the proof 
of Claim 2.
\end{description}
Let $r$ be a number with the property that
\[
r+1 \le \frac{n}{3} \mbox{ and } 2^{r+2}\cdot\binom{n}{r+1} \le  2^{n-r} 
\enspace .
\]
The condition $r+1 \le n/3$ implies 
that $\binom{n}{r+1} \ge \sum_{j=0}^{r}\binom{n}{j}$
so that $2^{r+2}\cdot\binom{n}{r+1} > \sum_{j=0}^{r}2^j\cdot\binom{n}{j}$.
An inspection of the definition of $r_*$ reveals that $r \le r_*$. We can 
therefore infer from Claim~1 that every sample in $\cS_{\le r}$ gets an 
$M$-partner. Setting $p := \frac{r+1}{n}$.  we obtain the following series 
of equivalent conditions:
\[
2^{r+2}\cdot\binom{n}{r+1} \le  2^{n-r} \dund \binom{n}{r+1} \le  2^{n-2r-2} 
\dund \binom{n}{pn} \le 2^{n-2pn} \dund \frac{1}{n}\log\binom{n}{pn} \le 1-2p 
\enspace .
\] 
It is well known that $\frac{1}{n}\log\binom{n}{pn}$ converges to $H(p)$
if $n$ goes to infinity. Let $p_* < 1/2$ be the unique solution of the
equation $H(p) =1-2p$. Fix any constant $0<p_0<p_*$. The theorem can now 
be obtained as follows. The inequality $p_0 < p_*$ implies
that $\frac{1}{n}\log\binom{n}{p_0 n} < 1 - 2 \cdot p_0$ provided
that $n$ is sufficiently large. This in turn implies that, 
for $r'_0$ given by $r'_0+1 = p_0 \cdot n$, we have
that $2^{r'_0+2}\cdot\binom{n}{r'_0+1} <  2^{n-r'_0}$. 
Setting $r_0 = \lfloor r'_0\rfloor$, 
we get $2^{r_0+2}\cdot\binom{n}{r_0+1} <  2^{n-r_0}$.
According to Claim 2, we have 
\[
\GMN(P_n) \le n-r_0-1 = n-\lfloor r'_0 \rfloor - 1 =
n - \lfloor r'_0+1 \rfloor = n - \lfloor p_0n \rfloor = \lceil (1-p_0)n \rceil
\enspace .
\]
This completes the proof of Theorem~\ref{th:powerset-gmn}.
\end{proof}

\section{Additivity Considerations} \label{sec:additivity}

\begin{definition} \label{def:free-combination}
Suppose that $C_1$ is a concept class over domain $X_1$, $C_2$ is
a concept class over domain $X_2$ and $X_1 \cap X_2 =  \eset$.
The \emph{free combination} of $C$ and $C'$, denoted 
by $C_1 \sqcup C_2$, is then given by
\[ 
C_1 \sqcup C_2 := \{c_1 \cup c_2: c_1 \in C_1\mbox{ and } c_2 \in C_2\}
\enspace .
\]
A combinatorial parameter $Y$ is said to be \emph{sub-additive}
(resp.~\emph{super-additive}) if 
\begin{equation} \label{eq:sub-additive}
Y(C_1 \sqcup C_2) \le Y(C_1) + Y(C_2)
\end{equation}
(resp.~$Y(C_1 \sqcup C_2) \ge Y(C_1) + Y(C_2)$) holds for each free combination
of two concept classes $C_1$ and $C_2$. If~(\ref{eq:sub-additive})
holds with equality for each free combination of two concept classes,
then $Y$ is said to be \emph{additive},
\end{definition}

\noindent
The goal of this section is to verify the entries of 
Table~\ref{table:additivity} below.

\begin{table}[h!]
\begin{center}
\begin{tabular}{|c||c|c|c|c|} 
\hline
& additive & sub-additive & super-additive & Reference \\
\hline
\hline
$\VCD$ & yes & yes & yes & common knowledge \\
\hline
$\RTD$ & yes & yes & yes & \cite{DFSZ2014} \\
\hline
$\min\{\VCD,\RTD\}$ & no & no & yes & this paper \\
\hline
$\STD_{min}$ & no & yes & no & this paper \\
\hline
$\AN$ & no & yes & no & this paper \\
\hline
$\AN'$ & no & yes & no & this paper \\ 
\hline
$\GMN'$ & no & yes & no & this paper \\
\hline
$\GMN$ & no & ? & no & this paper \\
\hline
$\SMN$ & no & yes & no & this paper \\
\hline
$\SMN'$ & no & yes & no & this paper \\
\hline
\end{tabular}
\caption{Additivity properties of the parameters of the
SMN-GMN-AMN hierarchy.}
\label{table:additivity}
\end{center}
\end{table}

\begin{remark}
$\VCD$ and $\RTD$ are known to be additive.
\end{remark}

The following relation, which holds for arbitrary reals $a,b,c,d$
is well known and easy to check:
\begin{equation} \label{eq:minimum-of-sums}
\min\{a+b , c+d\} \ge min\{a,c\} + \min\{b,d\} \enspace .
\end{equation}
Moreover, this inequality is strict if $a<c$ and $b>d$.

\begin{remark}
\begin{enumerate}
\item
If $Z_1$ and $Z_2$ are super-additive, then $Y := \min\{Z_1,Z_2\}$ is 
super-additive too.
\item
If $Z_1$ and $Z_2$ are additive and incomparable, then $Y := \min\{Z_1,Z_2\}$
is not additive.
\end{enumerate}
\end{remark}

\begin{proof}
Choose any two concept classes $C_1$ and $C_2$ over disjoint domains. 
Then:
\begin{eqnarray*}
Y(C_1 \sqcup C_2) & = & \min\{Z_1(C_1 \sqcup C_2) , Z_2(C_1 \sqcup C_2)\} \\
& \ge & \min\{Z_1(C_1) + Z_1(C_2) , Z_2(C_1) + Z_2(C_2)\} \\
& \stackrel{(*)}{\ge} & 
\min\{Z_1(C_1) , Z_2(C_1)\} + \min\{Z_1(C_2) , Z_2(C_2)\} = Y(C_1) + Y(C_2)
\enspace .
\end{eqnarray*}
The first equation holds by definition of $Y$. The subsequent inequality
holds by the super-additivity of $Z_1$ and $Z_2$. The inequality which 
is marked ``$(*)$'' is an application of~(\ref{eq:minimum-of-sums}). 
The final equation holds by definition of $Y$ again. The above calculation
shows that $Y$ is super-additive. Note that the inequality which is 
marked ``$(*)$'' becomes strict if $C_1$ and $C_2$ are chosen such 
that $Z_1(C_1) < Z_2(C_1)$ and $Z_2(C_2) < Z_1(C_2)$. Note that this 
is a possible choice under the assumption that $Z_1$ and $Z_2$ are 
incomparable. Thus, if $Z_1$ and $Z_2$ are incomparable, then $Y$
is not additive.
\end{proof}

\begin{example}
We may now conclude from the additivity and incomparability of $\VCD$ 
and $\RTD$ that $\min\{\VCD , \RTD\}$ is super-additive but not not 
additive.
\end{example}

\begin{remark} \label{rem:std-sub-additive}
$\STD_{min}$ is sub-additive but not additive.
\end{remark}

\begin{proof}
Choose any two concept classes $C_1$ and $C_2$ over disjoint 
domains $X_1 = \dom(C_1)$ and $X_2 = \dom(C_2)$.
For $i=1,2$, let $(M_k^i)_{k=0,\ldots,k^*(i)}$ be a subset teaching sequence 
of cost $d_i := \STD_{min}(C_i)$ for $C_i$. Let $M_{k_1,k_2}$ be the
$(C_1 \sqcup C_2)$-saturating matching given 
by $M_{k_1,k_2}(c_1 \cup c_2) = M_{k_1}(c_1) \cup M_{k_2}(c_2)$. 
It is easy to see that the sequence which starts with
$M_{0,0} , M_{1,0}1 , \ldots, M_{k^*(1),0}$ 
and continues with $M_{k^*(1),0} , M_{k^*(1),1} , \ldots , M_{k^*(1),k^*(2)}$
is a subset teaching sequence of cost $d_1+d_2$ for $C_1 \sqcup C_2$.
Hence $\STD_{min}$ is sub-additive. Let now $C_1 = \{c_0,c_1,c_2,c_3\}$ be 
the concept class over domain $X_1 = \{u,v\}$ given by $c_0 = \eset$,
$c_1 = \{u\}$, $c_2 = \{v\}$ and $c_3 = \{u,v\}$. Let $C_2$ consist of 
the single concept $X_2$ over domain $X_2 = \{a,b,c,d\}$. 
Then $\STD_{min}(C_1) = 2$ and $\STD_{min}(C_2) = 0$. The following 
subset teaching sequence shows that $\STD_{min}(C \sqcup C') = 1$:
\[
\begin{array}{ccl}
M_0(c_0 \cup X_2) & = & (u,0),(v,0),(a,1),(b,1),(c,1),(d,1) \\
M_0(c_1 \cup X_2) & = & (u,1),(v,0),(a,1),(b,1),(c,1),(d,1) \\
M_0(c_2 \cup X_2) & = & (u,0),(v,1),(a,1),(b,1),(c,1),(d,1) \\
M_0(c_1 \cup X_2) & = & (u,1),(v,1),(a,1),(b,1),(c,1),(d,1) \\
\hline
M_2(c_0 \cup X_2) & = & (u,0),(v,0),(a,1) \\
M_2(c_1 \cup X_2) & = & (u,1),(v,0),(b,1) \\
M_2(c_2 \cup X_2) & = & (u,0),(v,1),(c,1) \\
M_2(c_1 \cup X_2) & = & (u,1),(v,1),(d,1) \\
\hline
M_3(c_0 \cup X_2) & = & (a,1) \\
M_3(c_1 \cup X_2) & = & (b,1) \\
M_3(c_2 \cup X_2) & = & (c,1) \\
M_3(c_1 \cup X_2) & = & (d,1) 
\end{array}
\]
It follows that $\STD_{min}$ is not additive.
\end{proof}

The powerset $P_n$ over the domain $X = \{x_1,\ldots,x_n\}$ can be seen
as the free combination of $P_1^1 \sqcup\ldots\sqcup P_1^n$ where $P_1^i$
is the powerset over the domain $\{x_i\}$. For trivial reasons, we have
the following implications:
\begin{itemize}
\item If $Y$ is additive, then $Y(P_n) = n \cdot Y(P_1)$.
\item If $Y$ is sub-additive, then $Y(P_n) \le n \cdot Y(P_1)$.
\item If $Y$ is super-additive, then $Y(P_n) \ge n \cdot Y(P_1)$.
\end{itemize}

\begin{example} \label{ex:additivity-test}
It is easy to check that all parameters in the SMN-GMN-AMN hierarchy evaluate to $1$
on the class $P_1$. We know already that $\GMN(P_n)$ and $\AN(P_n)$ are 
less than $n$ provided that $n$ is sufficiently large. Of course, the same
must be true for all parameters below $\GMN$ and $\AN$ in the SMN-GMN-AMN hierarchy.
We may conclude from this short discussion that none of the 
parameters $\GMN , \SMN , \SMN' , \AN ,\AN'$ is super-additive 
(which implies that none of them is additive).
\end{example}

\begin{theorem} \label{th:sub-additive-parameters}
Each of the parameters $\GMN' , \SMN , \SMN' , \AN ,\AN'$ is sub-additive
but not additive.
\end{theorem}

\begin{proof}
Choose any two concept classes $C_1$ and $C_2$ over disjoint 
domains $X_1$ and $X_2$, respectively. Clearly
\[ |C_1 \sqcup C_2| = |C_1| \cdot |C_2| \enspace . \]
This multiplicativity  makes it easy to show sub-additivity 
for various parameters:
\begin{enumerate}
\item 
We first consider the parameter $\GMN'$.
For $i=1,2$, let $U_i = \binom{X_i}{\le d_i}$ be the set of unlabeled 
samples of size at most $d_i$ with sample points taken from $X_i$. 
Let $U = \binom{X_1 \cup X_2}{\le d_1+d_2}$ be the set of unlabeled 
samples of size at most $d_1+d_2$ with sample points taken from $X_1 \cup X_2$.
Then $(S_1,S_2) \mapsto S_1 \cup S_2$ is an injective mapping 
from $U_1 \times U_2$ to~$U$. Thus $|U| \ge |U_1| \cdot |U_2|$.
Hence, if $|U_1| \ge |C_1|$ and $|U_2| \ge |C_2|$, then 
\[
|U| \ge |U_1| \cdot |U_2| \ge |C_1| \cdot |C_2| = |C_1 \sqcup C_2|
\enspace ,
\]
which implies the sub-additivity of $\GMN'$.
\item
Let us now consider the parameter $\SMN'$. The reasoning is analogous 
to the reasoning for the parameter $\GMN'$. This time $U_i$ is the
number of $C_i$-realizable (labeled) samples of size at most $d_i$
and $U$ is the set of $(C_1 \sqcup C_2)$ realizable samples of size
at most $d_1+d_2$. The main observations is that the union of a
$C_1$-realizable sample with a $C_2$ realizable sample yields a
$(C_1 \sqcup C_2)$-realizable sample 
so that $(S_1,S_2) \mapsto S_1 \cup S_2$ is an injective mapping
from $U_1 \times U_2$ to $U$. Moreover, if $M_i$ is a $C_i$-saturating
matching in $(C_i,X_i)$ of cost $d_i$ (for $i = 1,2$), then $M$ given 
by $M(c_1 \cup c_2) = M_1(c_1) \cup M_2(c_2)$ 
is a $(C_1 \sqcup C_2)$-saturating matching. 
It follows that $\SMN'$ and $\SMN$ are sub-additive.
\item
The sub-additivity of the parameters $\AN'$ and $\AN$ can be shown
in a similar fashion. The only observation that we need, in addition
to the previous ones, is the following:
if $U_i \seq 2^{X_i \times \{0,1\}}$ is a collection of samples which 
forms an antichain (for $i=1,2$), 
then $U = \{S_1 \cup S_2: S_1 \in U_1 \mbox{ and } S_2 \in U_2\}$
is an antichain too. 
\end{enumerate}
We mentioned already in Example~\ref{ex:additivity-test} that
the parameters $\SMN , \SMN' , \AN , \AN'$ are not additive.
In order to complete the verification of the entries in 
Table~\ref{table:additivity}, we still have to show that $\GMN'$ is 
not additive. Choose both concept classs, $C_1$ and $C_2$, from the 
$(k,n)$-family with $k = n+2$, i.e., $|X_i| = n$ and $|C_i| = n+2$ 
for $i=1,2$. Then $\GMN'(C_1) = \GMN'(C_2) = 2$. A simple calculation 
shows that 
\[
\binom{2n}{0} + \binom{2n}{1} + \binom{2n}{2} + \binom{2n}{3} > (n+2)^2
= |C_1 \sqcup C_2|
\]
holds for all $n \ge 3$. Hence, assuming $n\ge3$, we have 
\[
\GMN'(C_1 \sqcup C_2) \le 3 < 4 = \GMN'(C_1) + \GMN'(C_2) \enspace ,
\]
which shows that $\GMN'$ is not additive. 
This complertes the proof of Theorem~\ref{th:sub-additive-parameters}.
\end{proof}

\section{A Final Look at the SMN-GMN-AMN Hierarchy}
\label{sec:final-look}

A \emph{combinatorial parameter (associated with concept classes)}
is a mapping $C \ra Y(C)$ which assigns a non-negative number $Y(C)$
to each (finite) concept class $C$. Suppose that $Y$ and $Z$ are two 
combinatorial parameters. We will write $Y \ra Z$ if the following 
hold:
\begin{enumerate}
\item 
For each concept class $C$, we have that $Y(C) \le Z(C)$.
\item 
The $\le$-relation between $Y(C)$ and $Z(C)$ is \emph{occasionally strict},
i.e., there exists a concept class $C$ such that $Y(C) < Z(C)$.
\end{enumerate}  
We say that $Y$ and $Z$ are \emph{incomparable} if there exist concept
classes $C_1$ and $C_2$ such that $Y(C_1) < Z(C_1)$ and $Z(C_2) < Y(C_2)$.
This is the case iff $Y \neq Z$ but neither $Y \ra Z$ nor $Z \ra Y$. 

In the course of this paper, we have verified various $\le$-relations
between combinatorial parameters associated with concept classes.
We will show now, by a series of examples and remarks, that all these 
$\le$-relations are occasionally strict so that we arrive at the diagram 
in Fig.~\ref{fig:hierarchy} below.

\begin{remark} \label{rem:smn-separation}
Suppose that $Y(C) \le Z(C)$ for each concept class $C$. Then,
if $Z$ is class monotonic but Y is not, then $Y \ra Z$.
Specifically, if $(C',X)$ is a class extension of $(C,X)$ 
such that $Y(C') < Y(C)$, then $Y(C') < Z(C')$.
\end{remark}

\begin{proof}
Let $(C',X)$ be a class extension of $(C,X)$ such that $Y(C') < Y(C)$.
Then $Z(C') \ge Z(C) \ge Y(C) > Y(C')$. 
\end{proof}

\begin{example}
We know already that $\SMN$ is class monotonic while $\SMN'$ is not.
Specifically, let $C = \vec{C}_{16,9}$ and let $C'$ be obtained from $C$
by adding the all-ones concept. Then, as we have already shown before,
we have that $\SMN(C') = 1 < \SMN(C')$. Now an application of
Remark~\ref{rem:smn-separation} yields $\SMN'(C') < \SMN(C')$. 
Hence $\SMN' \ra \SMN$.
\end{example}

\begin{example}
Clearly $\SMN(P_2) \ge 1$. The matching $M$ with 
\[
M(\eset) = \eset\ , M(\{x_1\}) = \{(x_2,0)\}\ ,\ M(\{x_2\}) = \{(x_2,1)\}\ ,\ 
M(\{x_1,x_2\}) = \{(x_1,1)\}
\]
witnesses that $\SMN(P_2) \le 1$. Thus $\SMN(P_2) = 1$. 
Clearly $\GMN(P_2) \le 2$. Consider the ordering 
$\{x_1,x_2\} , \{x_1\} , \{x_2\} , \eset$ for the concepts in $P_2$
and the ordering 
\[
\eset , \{(x_2,0)\} , \{(x_1,0)\} , \{(x_1,1)\} ,\{(x_2,1)\} , 
\{(x_1,0) ,(x_2,0)\} \ldots
\]
for the labeled samples over the domain $\{x_1,x_2\}$. The resulting 
greedy matching $M'$ is then given by
\[ 
M'(\{x_1,x_2\} = \eset\ ,\ M'(\{x_1\}) = \{(x_2,0)\}\ ,\ M'(\{x_2\} = \{(x_1,0)\}\ ,\ 
M'(\eset) = \{(x_1,0) , (x_2,0)\} \enspace .
\]
It follows that $\GMN(P_2) \ge 2$. Thus $\GMN(P_2) = 2 > 1 = \SMN(P_2)$
and, therefore, $\SMN \ra \GMN$. 
\end{example}

\begin{remark} \label{ex:powerset-gmn}
We know from Theorem~\ref{th:powerset-gmn} that $\GMN(P_n) < n$
for all sufficiently large $n$. We know furthermore that $\GMN'(P_n) = n$
for all $n \ge 1$. Hence $\GMN \ra \GMN'$.
\end{remark}

\begin{remark}
Suppose that $Y,Z,Z'$ are combinatorial parameters where $Z$ and $Z'$
are incomparable and, for each concept class $C$, we have that $Y(C) \le Z(C)$ 
and $Y(C) \le Z'(C)$. Then $Y \ra Z$ and $Y \ra Z'$. 
\end{remark}

\begin{proof}
Pick concept classes $C$ and $C'$ such that $Z(C) < Z'(C)$
and $Z'(C') < Z(C')$. Then $Y(C) \le Z(C) < Z(C')$
and $Y(C') \le Z'(C') < Z(C')$. Hence $Y \ra Z'$ and $Y \ra Z$.
\end{proof}

\begin{example}
Consider concept class $C = \vec{C}_{k,n}$ with $k = 2^\ell$
for some $\ell \ge 2$ and $n \ge k-1$. 
Since $\binom{n}{0} + \binom{n}{1} = 1 + n \ge k$,
it follows that $\GMN'(C) = 1$. It is easy to see 
that $\VCD(C) = \RTD(C) = \ell$. It follows 
that $\GMN'(C) < \min\{\VCD(C) , \RTD(C)\}$ and, 
therefore, $\GMN' \ra \min\{\VCD , \RTD\}$.
\end{example}

\begin{example} \label{ex1:smn-an}
The smallest number $d$ such that $\sum_{i=0}^{d}2^i\binom{6}{i} \ge 2^6$
equals $2$ whereas the smallest number $d$ such that $2^d\binom{6}{d} \ge 2^6$
equals $3$. It follows that $\SMN'(P_6) = 2 < 3 = \AN'(P_6)$.
Thus $\SMN' \ra \AN'$.
\end{example}

\begin{example} \label{ex:an}
We know already that $\AN$ is class monotonic while $\AN'$ is not.
Specifically, let $C = \vec{C}_{16,9}$ and let $C'$ be obtained from $C$
by adding the all-ones concept. Then, as we have already shown before,
we have that $\AN'(C') = 1 < \AN'(C)$. Now an application of
Remark~\ref{rem:smn-separation} yields $\AN'(C') < \AN(C')$.
Hence $\AN' \ra \AN$.
\end{example}

\begin{remark}
We mentioned already in Section~\ref{sec:powerset} 
that $\STD_{min}(P_n) = n$ is valid for all $n \ge 1$ 
and $\AN(P_n) < 0.23 n$ is valid for all sufficiently 
large $n$. It follows that $\AN \ra \STD_{min}$.
\end{remark}

\begin{remark} \label{rem:free-combination}
Suppose that $Z_1$ and $Z_2$ are incomparable additive combinatorial 
parameters and $Y$ is a sub-additive combinatorial parameter which 
satisfies $Y(C) \le Z_1(C)$ and $Y(C) \le Z_2(C)$ for each concept 
class $C$. Then $Y \ra \min\{Z_1,Z_2\}$.
\end{remark}

\begin{proof}
Pick two concept classes $C_1$ and $C_2$ such that $Z_1(C_1) < Z_2(C_1)$
and $Z_2(C_2) < Z_1(C_2)$. Rename the domain elements (if necessary) such 
that $\dom(C_1) \cap \dom(C_2) = \eset$. Then
\begin{eqnarray*}
Y(C_1 \sqcup C_2) & \le & Y(C_1) + Y(C_2) \le Z_1(C_1) + Z_2(C_2) \\
& < & 
\min\{Z_1(C_1)+Z_1(C_2) , Z_2(C_1)+Z_2(C_2)\} = 
\min\{Z_1(C_1 \sqcup C_2) , Z_2(C_1 \sqcup C_2)\}
\enspace.
\end{eqnarray*}
It follows that $Y \ra \min\{Z_1,Z_2\}$.
\end{proof}

\begin{example}
It is well known that $\VCD$ and $\RTD$ are additive.
According to Remark~\ref{rem:std-sub-additive}, $\STD_{min}$ 
is sub-additive. By an application of Remark~\ref{rem:free-combination}, 
we get $\STD_{min} \ra \min\{\VCD,\RTD\}$.
\end{example}

\begin{example}
We know already that $\AN(P_6) = \AN'(P_6) = 3$. We claim 
$\SMN(P_6) = 2$, which would imply that $\SMN \ra \AN$.
$\SMN(P_6) = 2$ can be shown by matching the empty set with the
empty sample, a singleton set $\{i\}$ by the sample $\{(i,1)\}$,
a co-singleton set $[6] \sm \{i\}$ by the sample $\{(i,0)\}$
and by matching the remaining concepts in $P_6$ by appropriately 
chosen samples of size $2$. We need to argue that such a saturating
matching for the remaining concepts does exist. To this end consider
the bipartite consistency graph (the graph with one vertex for
each concept $c$, one vertex for each labeled sample $S$ and an edge
between them iff $c$ is consistent with $S$. Each concept is consistent 
with $\binom{6}{2} = 15$ samples of size $2$. Each sample of
size $2$ is consistent with $2^4 = 16$ concepts from $P_6$,
including the empty set as well as the singleton and co-singleton sets. 
But, if we exclude these sets, then each sample of size $2$
is consistent with at most $10$ of the remaining concepts.
Now it is an easy application of Hall's theorem to show 
that all remaining concepts can be matched with appropriately
chosen samples of size $2$.
\end{example}

\begin{remark}
We mentioned already in Section~\ref{sec:powerset}  
that $\GMN'(P_n) = n$ is valid for all $n \ge 1$ 
and $\AN(P_n) < 0.23 n$ is valid for all sufficiently 
large $n$. Since $\AN'(P_n) = \AN(P_n)$, it follows 
that $\AN' \ra \GMN'$.
\end{remark}

\noindent
This concludes the verification of the diagram in Fig.~\ref{fig:hierarchy}.

\paragraph{Open Problem.}

We have left open the question of whether $\GMN$ is sub-additive.


\end{document}